\UseRawInputEncoding
\documentclass[11pt,reqno]{amsart}

\usepackage{amsmath,amsfonts,amsthm,amssymb,amsxtra}
\usepackage{hyperref}
\usepackage{color}
\usepackage[shortlabels]{enumitem}
\usepackage{bbm} 
\usepackage{stmaryrd}
\usepackage{nicematrix}
\usepackage{mathrsfs} 
\usepackage{float}
\usepackage{graphicx}

\usepackage[margin=1.1in]{geometry}
\usepackage{extarrows}

\usepackage{subcaption}

\usepackage{tikz}
\usepackage{pgfplots}
\pgfplotsset{compat=1.18} 
\usetikzlibrary{calc}
\newtheorem{theorem}{Theorem}[section]
\newtheorem{proposition}[theorem]{Proposition}
\newtheorem{lemma}[theorem]{Lemma}
\newtheorem{corollary}[theorem]{Corollary}
\theoremstyle{definition}

\theoremstyle{remark}
\newtheorem{remark}[theorem]{Remark}
\numberwithin{equation}{section}
\newcommand{\C}{\mathbb C}
\newcommand{\R}{\mathbb R}

\newcommand{\dd}{\,d}
\newcommand{\eps}{\varepsilon}
\newcommand{\scm}{\mu_{\mathrm{sc}}}

\DeclareMathOperator{\supp}{supp}
\DeclareMathOperator{\capc}{cap}
\DeclareMathOperator{\Lip}{Lip}

\DeclareMathOperator{\Int}{int}
\setlist[enumerate]{label=\textup{(\roman*)},leftmargin=*,itemsep=3pt,topsep=5pt}
\setlist[itemize]{leftmargin=*,itemsep=3pt,topsep=5pt}
\allowdisplaybreaks[1]
\title[Reflectionless measures on Widom sets]{Reflectionless measures with singular continuous components on Widom sets}
\author[L. Li]{Long Li}

\address{L. Li: Texas A\&M University,
College Station, TX 77843, USA}
\email{\href{mailto:longli@tamu.edu}{longli@tamu.edu}}

\begin{document}
\begin{abstract}
We implement a program of constructing a reflectionless measure on a Widom set with a nontrivial singular continuous component of Nazarov-Volberg and Yuditskii. The construction is an induction that based on two types of operations: splitting point mass and inserting bands. Each operation is realized by multiplying the Cauchy transform of the measure by some explicit multipliers. We provide necessary estimates on measures and the Green functions after each operation.
\end{abstract}
\maketitle

\section{Introduction and main result}\label{sec:introduction}

For a finite positive Radon measure $\nu$ supported on a compact subset $\mathsf{E}$ of the real line, define its {\it Cauchy transform}
\begin{equation}\label{eq:cauchy}
 \mathcal{C}_\nu(z)=\int_{\R}\frac{\dd\nu(t)}{t-z},\qquad z\in\C\setminus \mathsf{E}.
\end{equation}
We call the measure $\nu$ 
\emph{reflectionless} on $\mathsf{E}$\footnote{In \cite{NVY2007}, this definition was called weakly reflectionless. The almost every statement with respect to the measure $\nu$ itself leads to their definition of reflectionless.} when the finite boundary value
$\mathcal{C}_\nu(x+i0)$ exists and is purely imaginary for Lebesgue-almost every
$x\in \mathsf{E}$.

For a nonpolar compact set $\mathsf{E}\subset\R$, let $G_\mathsf{E}(z)=G_{\C\setminus \mathsf{E}}(z,\infty)$ be the Green function of the Denjoy domain $\Omega=\C\setminus \mathsf{E}$ with a pole at infinity.
When $\mathsf{E}$ is Dirichlet regular, we extend $G_\mathsf{E}$ continuously by zero to $\mathsf{E}$.
We call $\Omega$ a \emph{Widom domain} \cite{Widom1971Acta} if
\begin{equation}\label{eq:widom-definition}
 \sum_{I\in\mathcal G(\mathsf{E})}\max_{x\in I}G_\mathsf{E}(x)<\infty,
\end{equation}
where $\mathcal G(\mathsf{E})$ is the family of bounded complementary intervals
of $\mathsf{E}$. For subsets of the real line, this is the  sum of critical values of the
Green function.

\begin{theorem}\label{thm:main}
There exist a compact Dirichlet-regular set $\mathsf{E}\subset[-1,1]$, a compact
Lebesgue-null Cantor set $K\subset \mathsf{E}$, and a positive probability measure
$\mu$ with the following properties.
\begin{enumerate}
\item The set $\mathsf{E}$ is the union of $K$ and countably many nondegenerate
closed intervals, and $\supp\mu=\mathsf{E}$.
\item The domain $\C\setminus \mathsf{E}$ is Widom.
\item The measure $\mu$ is reflectionless on $\mathsf{E}$.
\item The measure $\mu$ is atomless, and its Lebesgue decomposition is
\begin{align}\label{eq:decomposition}
 \mu=w(x)\dd x+\scm,\qquad \scm=\mu|_K,\qquad \scm(K)>0.
\end{align}
The absolutely continuous part is also nonzero.
\end{enumerate}
Indeed, the construction below can be realized so that
$\mu(K)\ge \tfrac12 e^{-1/32}$.
\end{theorem}

Our interest in this work originates from the parameterization of isospectral reflectionless (Jacobi, Schr\"odinger, Dirac) operators \cite{Craig89,SY97,VY14,DY2016,BDGL18,FLLZ,DELL25,LiLukic2026} and canonical systems \cite{Yuditskii2011,BLY24} and the broader scope of Kotani theory \cite{Kotani1982,DS1983,Remling11,BLY22}. For a recent study on the density of absolutely continuous reflectionless measures, see \cite{WangZhou2026}. In the literature, reflectionlessness  lives on the support of the absolutely continuous spectral measure of an almost periodic operator where the Lyapunov exponent vanishes. In the context of operators, the definition of relectionlessness appears as vanishing of the real part of the diagonal Green's functions or equivalently complex conjugation of the half-line Titchmarsh-Weyl $m$ functions up to a sign.

Despite that for sufficiently nice $\mathsf{E} \subset\R$, reflectionlessness implies absence of singular components of the measure \cite{GY06,PR2009}, it remains an interesting question whether there exists a reflectionless measure with a nonempty singular (continuous) component. Before us,
Nazarov, Volberg and Yuditskii~\cite{NVY2007} proposed a construction in which
mass is repeatedly divided between two small intervals. The intended
singular component is supported on the limit set of the intervals. Yuditskii \cite{Yuditskii2011} constructed a reflectionless measure with a point mass.
However, the local splitting mechanism in \cite{NVY2007} needs to be supplemented by estimates that
hold at the end of each generation of operations, by control of every subsequent
change to an earlier result, and by a justification of boundary values
in the weak limit. The purpose of this paper is to give a complete
implementation of that program.

\emph{Idea of the construction:} The construction is inductive. An active atom is a point mass to be split. There are two elementary operations. The first operation is to replace an active atom by a band that centered at the atom and two child atoms. The resulting measure has positive density on the central band and point masses at child atoms. The second operation is to insert two guard bands that are symmetric about the two child atoms. The mass of the child atom will be reduced and redistributed to the guard bands, producing positive density on these bands. Then repeat the first operation on child atoms, then the second operation and so on. Eventually, every finite stage atom retires and the limiting measure is atomless. Each operation corresponds multiplying the Cauchy transform of the current measure by a factor given in Lemma \ref{lem:splitting} and Lemma \ref{lem:guards}. The inductive step and quantitative control are given by Proposition \ref{prop:induction}. To prove that the resulting domain is Widom type, we need certain control of the Green function, Lemma \ref{lem:screening} serves this purpose. We take the subordination principle of Green functions \cite[Section 4.4, Page 107]{Ransford} for granted in the analysis.

\section{Finite-stage transforms and local operations}\label{sec:operations}
In this section, we establish some basic results in order to control the measure distributions in each step of the operations. 

Given $r>0$, define the square root factor associated to a \emph{cut} $[-r,r]$ by
\begin{align}\label{eq:root}
    D_r(z)=\sqrt{z^2-r^2},\, D_r(z)/z\to 1 \text{ as } z\to\infty.
\end{align}
We choose the branch of the square root as follows. For $x\in\R$, $D_r(x)$ has the sign of $x$ for $|x|\geq r$ and $D_r(x+i0)=i\sqrt{r^2-x^2}$ for $|x|<r$. 
At every finite stage, the transforms used below have finitely many
interval cuts, integrable square-root singularities at their endpoints,
finitely many simple real poles, conjugation symmetry, and asymptotic expansion
$-1/z+O(z^{-2})$ at infinity. For a function $R(z)$ in this class, nonnegative
upper boundary imaginary parts on the cuts and nonpositive residues
at the poles imply that it is the Cauchy transform of a positive
probability measure. The Stieltjes inversion formula
gives density of the absolutely continuous component in the decomposition \eqref{eq:decomposition}
\begin{equation}\label{eq:inversion}
 w(x)=\frac1\pi\operatorname{Im}R(x+i0)
\end{equation}
on the cuts and mass $-\operatorname{Res}_{z=c}R$ at a pole $c$.
The difference between $R$ and this Cauchy transform has removable
singularities and vanishes at infinity. The coefficient of $-1/z$ in the asymptotic expression
gives the total mass one. 

Given an ordered
sequence of bands $[\alpha_j,\beta_j]$, $1\le j\le n$, define the canonical product as
\begin{equation}\label{eq:finite-product}
 R(z)=-\frac{\prod_{j=1}^{n-1}(z-\xi_j)}
 {\prod_{j=1}^{n}\sqrt{(z-\alpha_j)(z-\beta_j)}},
 \quad \xi_j\in[\beta_j,\alpha_{j+1}].
\end{equation}
Each root in the denominator is asymptotic to $z$ at infinity.
We allow the degenerate case $\alpha_j=\beta_j=c$, in which  the corresponding factor in the denominator
is $z-c$. 

The first operation replaces an atom by two atoms and a central band.

\begin{lemma}\label{lem:splitting}
Let $\nu$ be a given finite stage positive probability measure whose Cauchy transform is denoted by $R=\mathcal{C}_\nu$. Let $c$ be an isolated atom of $\mu$ and $h>0$ be sufficiently small such that the interval $(c-h,c+h)\setminus\{c\}$ does not intersect the support of $\mu$. Let $H(z)$ be a real analytic function such that on $(c-h,c+h)$ 
\begin{equation}\label{eq:local-H}R(z)=-\frac{H(z)}{z-c},\,H(c)=m=\nu(\{c\}).
\end{equation}
For each $0<\ell<h$, let
\begin{equation}\label{eq:split-factor}
 F_{c,h,\ell}(z)=
 \frac{(z-c)D_\ell(z-c)}{(z-c)^2-h^2}
\end{equation}
and denote $\widetilde R(z)=R(z)F_{c,h,\ell}(z).$
Then $\widetilde R(z)$ is the Cauchy transform of a positive probability measure denoted by $\widetilde\nu$.
It replaces the atom at $c$ by the central band $J=[c-\ell,c+\ell]$
and two child atoms at $c\pm h$ with masses
\begin{equation}\label{eq:child-masses}
 \widetilde m_\pm=
 \frac{H(c\pm h)}2\sqrt{1-(\ell/h)^2}.
\end{equation}
On $J$, its density of the absolutely continuous component is
\begin{equation}\label{eq:retired-density}
 \widetilde w(x)=\frac{H(x)}\pi
 \frac{\sqrt{\ell^2-(x-c)^2}}{h^2-(x-c)^2},
\end{equation}
and consequently
\begin{equation}\label{eq:leakage}
 \widetilde\nu(J)\le
 \frac{\max_J H}{2}\frac{\ell^2}{h^2-\ell^2}.
\end{equation}
\end{lemma}

\begin{proof}
Locally the new transform is
\[
 \widetilde R(z)=-\frac{H(z)D_\ell(z-c)}{(z-c)^2-h^2}.
\]
The residues at $c-h$ and $c+h$ give
\eqref{eq:child-masses}. Its upper boundary value $\lim_{\epsilon\to 0^+}\widetilde R(x+i\epsilon)$ on $J$ gives
\eqref{eq:retired-density}. Since
$h^2-(x-c)^2\ge h^2-\ell^2$ there and
$\int_{-\ell}^{\ell}\sqrt{\ell^2-u^2}\dd u=\pi\ell^2/2$,
integrating \eqref{eq:retired-density} gives \eqref{eq:leakage}.
For $|x-c|>h$, $F_{c,h,\ell}(x)>0$. Thus the old densities and
the other atomic masses remain positive. 
\end{proof}
\begin{remark}
Notice that if we denote $w(x)$ the density of $\nu$ restricted to $\R\setminus [c-h,c+h]$, then $\widetilde w(x)=w(x)F_{c,h,\ell}(x)$. Since $F(c,h,\ell)(x)>0$ on $\R\setminus [c-h,c+h]$, $\nu$ and $\widetilde\nu$ have the same support outside $[c-h,c+h]$. Because $F_{c,h,\ell}(\infty)=1$, the total mass remains to be one.
\end{remark}
Following the first operation immediately, we introduce the second operation. This inserts two small bands symmetric about each child atom generated by the first operation, and is realized by multiplying the density function by the factors given below. Measures are redistributed among the child atoms and the newly inserted bands.
\begin{lemma}\label{lem:guards}
Fix an atom at $c,$ let $H(z),R(z)$ be the functions given by \eqref{eq:local-H}.
For $k\in(0,1)$, choose $a>0$ sufficiently small so that $H>0$ on $[c-a,c+a]$ the interval contains no other support.
Denote the factor
\begin{equation}\label{eq:guard-factor}
 \Gamma_{c,a,k}(z)=\frac{D_{ka}(z-c)}{D_a(z-c)}.
\end{equation}
Then the product $R(z)\Gamma_{c,a,k}(z)$ is the Cauchy transform of a positive probability measure, which gives a point mass $km$ at $c$ and has positive densities on the following intervals 
\begin{equation}\label{eq:guard-bands}
 L=[c-a,c-ka],\qquad Q=[c+ka,c+a].
\end{equation}
 Moreover, $\Gamma_{c,a,k}$ extends analytically through the interval $(c-ka,c+ka)$.
\end{lemma}

\begin{proof}
One can verify that the factor $\Gamma_{c,a,k}$ is positive on $(c-ka,c+ka)$ and that $\Gamma_{c,a,k}(c)=k$ by direct computations. On $L$, $D_{ka}(z-c)>0$ and $D_a(z-c)$ is purely imaginary with a negative part. Therefore, $\Im \Gamma_{c,a,k}|_L>0.$ Since $R(x)$ is positive there, the product $R\Gamma_{c,a,k}>0$ on $L$. The analysis on $Q$ is similar.

For $z<c-a$ and $z>c+a$, $\Gamma_{c,a,k}>0$, preserving the density on the original support.
\end{proof}
The next lemma explains the continuity of each operation as the corresponding scales shrink to zero.
\begin{lemma}\label{lem:continuity}
For fixed $0<\lambda<1$ and $0<k<1$,
\begin{equation}\label{eq:multiplier-continuity}
 F_{c,h,\lambda h}\to 1\quad(h\downarrow0),\qquad
 \Gamma_{c,a,k}\to 1\quad(a\downarrow0)
\end{equation}
locally uniformly on $\C\setminus\{c\}$.
If all measures are supported in a fixed compact interval, each
operation consequently converges weakly to the input measure as its
scale tends to zero. It also changes every other isolated atomic mass
by a factor tending to one. The masses of any fixed intervals whose
boundaries lie outside the input support converge to their input masses.
\end{lemma}

\begin{proof}
Let $F\subset\C\setminus\{c\}$. By the definition of $F_{c,h,\ell}$ \eqref{eq:split-factor} and $\Gamma_{c,a,k}$ \eqref{eq:guard-factor}, the uniform convergence on $F$ of $F_{c,h,\lambda h}\to 1$ as $h\to 0^+$ and $\Gamma_{c,a,k}\to 1$ as $a\to 0^+$ follows.

For a given measure $\mu$, let $R(z)=\mathcal{C}_{\mu}$ be its Cauchy transform. Let $\widetilde{R}=R(z)\mathcal{Q}(z)$, where $\mathcal{Q}$ denotes either the split-factor $F$ or the guard-factor $\Gamma$. The limit will then be taken with respect to $h$ and $a$ respectively. As the scale converge to $0$, we have uniformly on $F$, $\widetilde{R}\to R$. Let $\widetilde{\mu}$ be the measure associated to $\widetilde{R}$, since the probability measure on $[-1,1]$ form a compact set, there is a unique weak limit $\mu'$ where $\widetilde{\mu}$ converges to along a subsequence. Necessarily, the measure $\mu'$ is uniquely determined by $R$. Therefore, $\mu'=\mu.$ The convergence of point masses and interval masses is immediate.
\end{proof}

\subsection{Endpoint-divisor bookkeeping}\label{subsec:divisors}
It is helpful to explain how the canonical product representation \eqref{eq:finite-product} is preserved after each operation. Assume that $R(z)$ is already in the form \eqref{eq:finite-product} and is associated to a finite stage measure $\mu.$ Let $c$ be an active atom, then the central band replacement and atom splitting in Lemma ~\ref{lem:splitting} replace the factor $\frac{1}{z-c}$ by 
\[\frac{(z-c-\ell)(z-c+\ell)}{\sqrt{(z-c-\ell)(z-c+\ell)}\sqrt{(z-c-h)(z-c-h)}(z-c+h)}.\]
Therefore the expression \eqref{eq:finite-product} (with the possibility that $\xi_j\in\{\beta_j,\alpha_{j+1}\}$) is preserved by this operation.
The guards insertion contributes the factor \[\Gamma_{c,a,k}=\frac{(z-c-ka)(z-c+ka)}{\sqrt{(z-c-ka)(z-c+ka)}\sqrt{(z-c-a)(z-c+a)}},\]
which also preserves \eqref{eq:finite-product}. The central bands and inserted guard bands are finalized right after their birth. Further operations will not affect them.

\section{Uniform screening by thin intervals}\label{sec:screening}
Let $F\subset \R$ be a compact regular set and let $G_F(z)$ be the Green function of the unbounded domain $\Omega=\C\setminus F$ with a  pole at infinity. Then  
\begin{equation}\label{eq:green-normalization}
 G_F(z)=\log|z|-\log\capc(F)+o(1)\qquad(z\to\infty).
\end{equation}
Basic facts about logarithmic capacity, Green functions, and regularity
may be found in~\cite{Ransford}. We will need an estimate of the Green function on the interval (corridor) between the two inserted guards for an atom.
\begin{lemma}\label{lem:screening}
Suppose $F$ is a finite union of nondegenerate compact intervals and $c\notin F$.
For any given $k\in(0,1)$, let
\[
 T_k=[-1,-k]\cup[k,1],\qquad T_a=c+aT_k.
\]
Then as $a\to 0^+$, we have
\begin{equation}\label{eq:screening-limit}
 \sup_{|x-c|\le ka}G_{F\cup T_a}(x)\to 0
\end{equation}
\end{lemma}

\begin{proof}
Since $F$ is closed and $c\notin F$, there is $\delta>0$ so that the closed disc $B(c,\delta)$ is disjoint with $F.$ Let $G_F(z),G_{T_k}(z)$ be the associated Green functions of $F$ and $T_k$ respectively. let $M=\max_{|z-c|=r}G_F(z)$ and $L_a=\min_{|\zeta|=r/a}G_{T_k}(\zeta)$. It follows from \eqref{eq:green-normalization} that $L_a\to\infty$ as $a\to 0^+$.
Consider the function $z\mapsto \frac{M}{L_a}G_{T_k}((z-c)/a)$. For $z\in T_a$, we have $(z-c)/a\in T_k$ and $G_{T_k}|_{T_a}((z-c)/a)=0$, and obviously $G_{F\cup T_a}|_{T_a}(z)=0$. On $|z-c|=r$, we have $\frac{M}{L_a}G_{T_k}((z-c)/a)\geq M$ and  by the subordination principle,
\[G_{F\cup T_a}\leq G_F\leq M.\]
It follows from the maximum principle that the following inequality holds
\begin{equation}\label{eq:screening-bound}\sup_{|x-c|\leq ka}G_{F\cup T_a}(x)\leq \frac{M}{L_a}\max_{|t|\leq k}G_{T_k}(t).
\end{equation}
Since $L_a\to \infty$ as $a\to 0^+$, we have \[\lim_{a\to 0^+}\left(\sup_{|x-c|\leq ka}G_{F\cup T_a}(x)\right)=0.\] 
\end{proof}

\begin{remark}\label{rem:polar}
By the Green function of a finite stage in our construction we mean the Green function associated to the finite union of bands. The active atoms in each finite generation are of course polar and we do not require the Green function to vanish there.
\end{remark}

\section{The inductive construction}\label{sec:induction}

Let $\mathcal T=\bigcup_{n\ge0}\{0,1\}^n$ be the binary tree, whose
root vertex is the empty word $\varnothing$. Write $|v|$ for depth/length and $v0,v1$ for the children of $v$. A \emph{permanent band} is an interval retained in every subsequent support. An \emph{active atom} is an isolated point mass to be split in the following generation.

Associated to each vertex $v$, the active atom is denoted by $c_v$. Each $c_v$ possesses two guards (bands) that are symmetric about it. Its \emph{corridor} denoted by $P_v$ is the open interval between the two guards. A pocket of $c_v$ is denoted by $N_v$. It lies inside the corridor $P_v$ with much smaller size. We will be more specific below.

We equip the space of probability measures on $[-1,1]$ with the  Kantorovich--Rubinstein metric
\begin{equation}\label{eq:BL}
 d_{KR}(\sigma,\tau)=
 \sup_{\substack{\|f\|_\infty\le1\\\Lip(f)\le1}}
 \left|\int f\dd\sigma-\int f\dd\tau\right|,
\end{equation}
which induces the weak topology. Fix the error sequences
\begin{equation}\label{eq:budgets}
 \eps_n=2^{-n-4},\quad \delta_n=4^{-n-3}\quad(n\ge0),
 \qquad \gamma_n=2^{-3n}\quad(n\ge1).
\end{equation}
In particular,
\begin{equation}\label{eq:budget-sums}
 S:=\sum_{n\ge0}\eps_n=\frac{1}{8},\qquad
 \sum_{j\geq n}\delta_j=\frac{4^{-n}}{48},\qquad
 \sum_{n\geq 1}2^n\gamma_n<\infty.
\end{equation}

\begin{proposition}\label{prop:induction}
There are finite-stage transforms $R_n=\mathcal{C}_{\mu_n}$, unions $B_n$ of
permanent bands, and exactly $2^n$ active atoms $c_v$, $|v|=n$, with
masses $m_v=\mu_n\{c_v\}$, having the following properties.
\begin{enumerate}
\item Every $\mu_n$ is a positive probability meaure on $[-1,1]$, its support
is $B_n\cup\{c_v:|v|=n\}$, and the finite endpoint-divisor invariance
of Section~\ref{subsec:divisors} holds.
\item The pockets of parent and children satisfy
\begin{equation}\label{eq:pocket-geometry}
 N_{v0}\cap N_{v1}=\varnothing,\qquad
 N_{v0}\cup N_{v1}\Subset\Int N_v,\qquad
 |N_v|\le2^{-2|v|}.
\end{equation}
At its birth, $N_v$ contains only the atom $c_v$. All modifications
at $v$ lie in the interior of the middle half of $N_v$.
\item Each permanent band $J$ has a compact complex neighborhood
$D_J$ containing $J$ in its interior and disjoint from all active
pockets at the end of its birth generation. It is also disjoint
from the other support components at that generation.
\item At the end of the transition $n\to n+1$,
\begin{align}
 \frac{1}{2}e^{-\delta_n}m_v&\leq m_{vi}\leq
 \frac{1}{2}e^{\delta_n}m_v\quad(i=0,1),\label{eq:mass-invariant}\\
 d_{KR}(\mu_{n+1},\mu_n)&\leq\epsilon_n,\label{eq:weak-invariant}\\
 |\mu_{n+1}(N_w)-\mu_n(N_w)|&\leq\frac{\delta_n}{2}
 \quad(|w|\leq n),\label{eq:cylinder-invariant}\\
 \sup_{D_J}|R_{n+1}/R_n-1|&\leq\delta_n
 \quad(J\subset B_n).\label{eq:analytic-invariant}
\end{align}
In \eqref{eq:analytic-invariant}, the quotient denotes the explicit
product of the local multipliers, which is analytic on $D_J$.
\item The Green functions of the permanent bands satisfy
\begin{equation}\label{eq:green-invariant}
 \sup_{x\in P_v}G_{B_n}(x)\le\gamma_n
 \quad(|v|=n\ge1).
\end{equation}
\end{enumerate}
The root atom can be taken to have mass $q=1/2$.
\end{proposition}

\begin{proof}
\emph{The initial step:}
let us first set $q=1/2$ and take
\begin{equation}\label{eq:initial}
 R_0(z)=-\frac{D_q(z)}{zD_1(z)},\quad
 B_0=[-1,-q]\cup[q,1].
\end{equation}
This is the transform of a positive probability measure  with an atom of mass $q$
at zero and density on $B_0$
\begin{equation}\label{eq:initial-density}
 w_0(x)=\frac1\pi\frac{\sqrt{x^2-q^2}}
 {|x|\sqrt{1-x^2}},\quad q<|x|<1.
\end{equation}
The two initial intervals are the root guards in $B_0$, and the initial corridor
$P_\varnothing=(-q,q)$. Choose a small closed root pocket $N_\varnothing$ centered at
zero and compactly contained in $P_\varnothing$, and choose protected
neighborhoods of the two root guards $[-1,-q]$ and $[q,1]$ disjoint from $N_{\varnothing}$.

\emph{The inductive step:} suppose generation $n$ has been constructed and let $v$ be a vertex so that $|v|=n$. At $v$, let
$N_v=[c_v-r_v,c_v+r_v]$ be the pocket. Let
\begin{equation}\label{eq:local-geometry}
 0<h_v<r_v/4,\qquad 0<\ell_v<h_v/4.
\end{equation}
Lemma~\ref{lem:splitting} creates the central band $J_v=[c_v-\ell_v,c_v+\ell_v]$
and the two child atoms $c_{v0}=c_v-h_v$, $c_{v1}=c_v+h_v$.
Install guards of the form \eqref{eq:guard-bands} about each child using Lemma \ref{lem:guards}, of radii $a_{vi}<h_v/4$ and
ratios $k_{vi}\in(0,1)$. All five new permanent bands (one central, four guards of children) are disjoint,
miss every other support component, and lie in the middle half of
the pocket $N_v$, after decreasing the scales as needed. The child corridor of the atom $c_{vi}$ is
$P_{vi}=(c_{vi}-k_{vi}a_{vi},c_{vi}+k_{vi}a_{vi}), i=0,1.$
At the end of this generation choose a pocket $N_{vi}$ of $c_{vi}$ with radius less than
$k_{vi}a_{vi}/2$, and shrink it further (by decreasing $a_{vi}$) if necessary to satisfy
\eqref{eq:pocket-geometry}.

There are $M_n=3\cdot2^n$ operations in $n$th generation:
one atom splitting and two guard insertions per parent $v$ with $|v|=n$. Before shrinking a
split scale, fix $\lambda_v=\ell_v/h_v\in(0,1/4)$ so small that 
 $\sqrt{1-\lambda_v^2}\ge e^{-\delta_n/4}.$
Fix each guard ratio such that $k_{vi}\ge e^{-\delta_n/4}$ and $k_{vi}<1$.
At the time of a split, shrink $h_v$ if necessary so that the current analytic
factor in \eqref{eq:local-H} satisfies
\begin{align}\label{eq:H-Factor}
 e^{-\delta_n/4}\leq
 \frac{H(c_v\pm h_v)}{H(c_v)}\le e^{\delta_n/4}.
\end{align}
In addition, for every  multiplier $\mathcal{Q}$ associated to each elementary operation of each atom $c_v$ in this generation, require
\begin{equation}\label{eq:external-atom-budget}
 |\log \mathcal{Q}|\leq \frac{\delta_n}{5M_n}
\end{equation}
locally uniformly on $\C\setminus \{c_v\}$. Since there are $2^n$ atoms at the generation of $v$, by Lemma~\ref{lem:continuity}, the requirement \eqref{eq:external-atom-budget} can fulfilled by choosing the scales $h_v,a_v$ sufficiently small. Fix a pocket $N_v$ for $v$, operations applied to other pockets contributes a multiplier of absolute logarithmic bounded by $M_n\sum_{|v|=n}|\log\mathcal{Q}|\leq \frac{\delta_n}{5}$. This estimate is uniform across $v$ in generation $n.$

To prove \eqref{eq:mass-invariant}, we track the mass change after all operations in the current generation. The mass $m_{vi}$ assigned to $vi$ at its birth is given by \eqref{eq:child-masses}. Following \eqref{eq:H-Factor} and Lemma~\ref{lem:guards}, $e^{-\delta_n/4}\leq \sqrt{1-\lambda_v^2}<1, e^{-\delta_n/4}\leq k_{vi}<1$, and the total contribution with absolute logarithmic bound $\delta_n/4$ from the operations applied to other pockets, we have \eqref{eq:mass-invariant}. The point mass splitting can be made better, but \eqref{eq:mass-invariant} is sufficient for our purpose.

We next verify \eqref{eq:weak-invariant} and \eqref{eq:cylinder-invariant}. We first notice that all atoms in previous generations retire in the current generation, hence $\mu_n$ and $\mu_{n+1}$ give zero mass to those retired atoms. Let $B_{n-1}$ be the union of permanent bands of the previous generation, since the densities of $\mu_n$ and $\mu_{n+1}$ restricted to $B_{n-1}$ only differ by at most $\delta_n/4$, and the support is contained in $[-1,1]$, so the masses of $B_n$  assigned by $\mu_n$ and $\mu_{n+1}$ can differ at most $\delta_n/2$. It remains to estimate the masses for $|w|=|v|=n$ and $N_w$. Our choice of scales $h_v,\lambda_v$ can be made so small that at the same time $\frac{1}{2}e^{-\delta_n/5}\mu_w\leq\mu_{wi}\leq \frac{1}{2}e^{\delta_n/5}\mu_{w}$. It follows that $|\mu_{n+1}(N_w)-\mu_n(N_w)|\leq \frac{\delta_n}{2}$. This implies already \eqref{eq:cylinder-invariant}. We are overdoing this in order to establish \eqref{eq:weak-invariant}. For any $f$ with $\mathrm{Lip}(f)\leq 1$ and $\Vert f\Vert_\infty \leq 1,$ we compute
\begin{align*}\left|\int_{[-1,1]}fd\mu_{n}-\int_{[-1,1]}fd\mu_{n+1}\right|&\leq \left|\int_{B_{n-1}}fd\mu_{n}-\int_{B_{n-1}}fd\mu_{n+1}\right|\\&\quad+\left|\sum_{|v|=1}\int_{N_v}fd\mu_n-\int_{N_v}fd\mu_{n+1}\right|\\
&\leq \frac{\delta_n}{2}+\sum_{|v|=1}\left(|\mu_n(N_v)-\mu_{n+1}(N_v)|+|N_v|\mu_{n+1}(N_v)\right)\\
&\leq\frac{\delta_n}{2}+2^n\delta_n+4^{-n}\leq\epsilon_n.
\end{align*}

The estimate \eqref{eq:analytic-invariant} follows from \eqref{eq:external-atom-budget} and the following
\[
 \sup_{D_J}\left|\prod_{M_n} \mathcal{Q}-1\right|
 \le\left(1+\frac{\delta_n}{2M_n}\right)^{M_n}-1
 \le e^{\delta_n/2}-1\le\delta_n.
\]

At a child's guard insertion, apply Lemma~\ref{lem:screening} to the
current union of permanent bands. The ratio $k_{vi}<1$ is already
fixed, but $a_{vi}$ can still be arbitrarily small. Thus it can also
be chosen so that the Green function is at most $\gamma_{n+1}$
throughout the corridor $P_{vi}$. Additional bands only decrease that Green function,
so \eqref{eq:green-invariant} holds at the end of the generation.
The isolated atoms are treated as in Remark~\ref{rem:polar}.
\end{proof}

\section{The limiting measure}\label{sec:limit}
For the binary tree $\mathcal{T}$, let $v\in \mathcal{T}$ denote a vertex and let $N_v$ be the  pocket associated to the atom $c_v$.
Define the sets
\begin{equation}\label{eq:limit-sets}
 A_n=\bigcup_{|v|=n}N_v,\qquad
 K=\bigcap_{n\ge0}A_n,\qquad
 \mathsf{E}=\overline{\bigcup_{n\ge0}B_n}.
\end{equation}
Then  $A_n$ are decreasing and compact, and one can verify that
\begin{equation}\label{eq:null-set}
 |A_n|\le2^{-n},\qquad |K|=0.
\end{equation}
By the construction of the pockets, $|N_v|\leq 2^{-2|v|}$. It follows that $K$ is a perfect set with Lebesgue measure zero. Therefore, $K$ is a Cantor set. Since each band is finalized at certain finite generation, it is disjoint with $K.$

\begin{lemma}\label{lem:limit-support}
Let $\mu_n$ be the finite stage probability measure at generation $n$. Then there exists a probability measure $\mu$ supported on $\mathsf{E}$ such that $\mu_n$ converge weakly to $\mu$. Moreover,
\begin{equation}\label{eq:E-description}
 \mathsf{E}=K\cup\bigcup_{n\ge0}B_n.
\end{equation}
Every pocket boundary has an open neighborhood disjoint from $\mathsf{E}$.
\end{lemma}

\begin{proof}
Since $d_{KR}(\mu_n,\mu_{n+1})<\delta_n$, $\{\mu_n\}$ is a Cauchy sequence. The set of probability measures supported on the bounded closed interval $[-1,1]$ is compact, therefore, $\mu_n$ converges weakly to a unique probability measure $\mu$. For each finite stage measure, $\mu_n$ its atom is contained in a permanent band born at the next generation. Therefore, $\mathrm{supp}\mu_n\subset \mathsf{E}$. It follows that $\mathrm{supp}\mu\subset \mathsf{E}$.

To show that $\mathsf{E}=K\cup\bigcup_{n\geq 0}B_n$, we show that every $x\in K$ is approached by permanent bands. Since $A_n$ is the union of generation $n$ pockets and only finitely many permanent bands lie outside $A_n$. Those bands generated in future generations all lie inside $A_n$. If $x\in \mathsf{E}$ but $x\notin K$, then it belongs to at most finitely many $A_n$ and therefore contained in some permanent band. On the other hand, if a point $x\in K$, then it does not belong to any permanent band. Moreover, there is a nested sequence of pockets whose intersection equals $x.$ Since each pocket contains a central replacement band, $x$ is approached by $\bigcup_{n\geq 0}B_n$. Therefore \eqref{eq:E-description} follows.

For each fixed pocket $N_v$, it lies in the central half of a corridor $P_v$, and future operations are confined to the central half of $N_v$, it follows that each boundary point of $N_v$ admits a collar neighborhood that is disjoint from $\mathsf{E}$.
\end{proof}
The following two propositions show that the set $K$ only supports the singular continuous component of the limiting measure $\mu.$
\begin{proposition}\label{prop:singular-mass}
The limiting measure satisfies
\begin{equation}\label{eq:positive-singular}
 \mu(K)\ge q e^{-S}=\tfrac12e^{-1/8}>0.
\end{equation}
\end{proposition}

\begin{proof}
Repeated use of \eqref{eq:mass-invariant} gives
\begin{equation}\label{eq:tree-mass-bounds}
 m_v\le q e^S2^{-|v|},\qquad
 \sum_{|v|=n}m_v\ge q e^{-S}=:c_*.
\end{equation}
For a fixed $n$ and every $j\ge n$, all generation-$j$ active atoms lie in
$A_n$. Hence $\mu_j(A_n)\ge c_*$. As $A_n$ is closed, the Portmanteau theorem (the ``$\limsup$ bound'' case)
gives
\[
 \mu(A_n)\ge\limsup_{j\to\infty}\mu_j(A_n)\ge c_*.
\]
It follows from the upper semicontinuity that $\mu(K)\ge c_*$. 
\end{proof}

\begin{proposition}\label{prop:no-atoms-K}
Let $v$ be a vertex of $\mathcal{T}$ at generation $n$ and let $N_v$ be the associated pocket of the atom $c_v$. Then we have
\begin{equation}\label{eq:upper-cylinder}
    \mu(N_v)\leq qe^{S}2^{-n}+\frac{4^{-n}}{48}.
\end{equation}
As a consequence, $\mu$ is atomless on $K$.

\end{proposition}
\begin{proof}
    Notice that $\mu_n(v)=m_v$. By \eqref{eq:mass-invariant}, by induction we have $\mu_n(v)\leq qe^S2^{-n}$. Since $\mu(N_v)=\sum_{k=n}^\infty (\mu_{k+1}(N_v)-\mu_k(N_v))+\mu_n$, by the estimate \eqref{eq:cylinder-invariant}, the result follows. 

    By the definition of $K$, for each $x\in K$, for any $n$, there exists a $v\in\mathcal{T}$ at generation $n=|v|$ such that $x\in N_v$. Then the estimates \eqref{eq:upper-cylinder} implies that $\mu(x)=0$. Thus $\mu|K$ is atomless.
\end{proof}

\begin{proposition}\label{prop:bands}
Let $\mu$ be the limiting measure and $\mathcal{C}_\mu$ be its Cauchy transform. For each permanent band $J$, $\mu$ is purely absolutely continuous on the interior of $J$ and has no point mass at the end points of $J$. Moreover, $0<C_{\mu}(x+i0)=\mathrm{Im} \mathcal{C}_\mu(x+i0)<\infty$ for every $x\in J$.
\end{proposition}

\begin{proof}
Let $b$ be the birth generation of the band $J$ and denote the associated finite stage density function $R_b(z)$. Let $P_J(z)=\prod_{n\geq b}\frac{R_{n+1}(z)}{R_{n}(z)}$. By the estimate \eqref{eq:analytic-invariant}, $P_J(z)$ converges uniformly to an analytic function on the protected neighborhood $D_J$ of $J.$ Moreover, since each factor is positive on $J,$ we have $P_J(x)>0$ on $J.$ It follows from the weak convergence $\mu_n$ to $\mu$ that locally uniformly off $J$
\[C_{\mu}(z)=R_b(z)P_J(z).\]
Since $0<R_b(x+i0)=\mathrm{Im} R_b(x+i0)<\infty$ on the interior of $J$ and $0<P_J(x)<\infty$ on $J$, it follows that $0<\mathcal{C}_\mu(x+i0)=\mathrm{Im} \mathcal{C}_{\mu}(x+i0)<\infty$ on the interior of $J.$
Therefore, $\mu$ is reflectionless and purely absolutely continuous on the interior of $J$.

Let $e$ be an end point of $J$. Since $\mathcal{C}_\mu=R_bP_J$ and locally near $e$, $R_b$ grows at most at the rate of $O(|z-e|^{-1/2})$, it follows that
\[\mu(\{e\})=\lim_{y\to 0^+}\mathrm{Im} \mathcal{C}_\mu(e+iy)=0.\]
Thus $\mu$ has no point mass at the end points of $J.$
\end{proof}
We are now ready to prove the statements (i), (iii) and (iv) of Theorem \ref{thm:main}. The proof of statement (ii) will be left to next section.
\begin{corollary}\label{cor:measure-conclusions}
The limiting measure $\mu$ is atomless, reflectionless on $\mathsf{E}$, and
satisfies \eqref{eq:decomposition}. Both terms in that decomposition
are nonzero and $\supp\mu=\mathsf{E}$.
\end{corollary}

\begin{proof}
Since the limiting measure is continuous on permanent bands and that $\mu|_{K}$ is atomless by Proposition~\ref{prop:no-atoms-K}, it follows that $\mu$ is atomless on $\mathsf{E}$.

To show that $\mu$ is reflectionless on $\mathsf{E}$, it suffices to show that $\mathcal{C}_\mu(x+i0)=0$ for Lebesgue-almost every $x$ on the permanent bands, since $|K|=0.$ This follows from Proposition~\ref{prop:bands}. The fact that $\mu$ has nontrivial absolutely continuous and singular continuous components follows from Proposition~\ref{prop:bands} and Proposition~\ref{prop:singular-mass} respectively.
\end{proof}

\section{Regularity and Widom summability}\label{sec:widom}

For a nonroot vertex $v$ of length $n$, subordinate principle of Green function and
\eqref{eq:green-invariant} imply
\begin{equation}\label{eq:limit-green-bound}
 G_\mathsf{E}(x)\le\gamma_n\qquad(x\in P_v\setminus \mathsf{E}).
\end{equation}
At this point no extension by zero across unknown irregular points
is being used. The root corridor has a finite majorant
\begin{equation}\label{eq:root-G}
 B:=\max_{[-q,q]}G_{B_0}<\infty.
\end{equation}

\begin{lemma}\label{lem:regularity}
Let $\mathsf{E}$ be the set in \eqref{eq:E-description}, then $\Omega=\C\setminus \mathsf{E}$ is Dirichlet regular.
\end{lemma}

\begin{proof}
A point of a permanent interval is regular by comparison with the
Green function of that interval. If $x\in K$, let $v$ be its
depth-$n$ vertex. The finite union $B_n$ has no band in $P_v$, and
$G_{B_n}$ is harmonic near $x$, including when $x$ is an isolated
active point of the finite-stage support. It is a continuous majorant
of $G_\mathsf{E}$ off $\mathsf{E}$, and $G_{B_n}(x)\le\gamma_n$. Consequently
\[
 \limsup_{\substack{z\to x\\z\notin E}}G_\mathsf{E}(z)\le\gamma_n.
\]
Letting $n\to\infty$ proves regularity. 
\end{proof}

The following observation is obvious, see Figure~\ref{fig:split}. We omit its proof.
\begin{lemma}\label{lem:gap-count}
For every $v\in\mathcal{T}$, let $c_v$ be the atom associated to $v$. At the current generation, the operations on $c_v$ creates four finalized bounded gaps contained in the associated corridor $P_v$. On the other hand, every gap is finalized at exactly one vertex.
\end{lemma}

\begin{figure}[tb]
\centering
\begin{tikzpicture}[x=0.85cm,y=0.72cm,font=\small]
 \draw[thin] (0,0)--(12,0);
 \foreach \a/\b in {0/0.8,2/2.6,3.9/4.5,5.6/6.4,7.5/8.1,9.4/10,11.2/12}
   {\draw[line width=3.2pt] (\a,0)--(\b,0);}
 \node[above] at (0.4,0.12) {$L_v$};
 \node[above] at (2.3,0.12) {$L_{v0}$};
 \node[above] at (4.2,0.12) {$Q_{v0}$};
 \node[above] at (6,0.12) {$J_v$};
 \node[above] at (7.8,0.12) {$L_{v1}$};
 \node[above] at (9.7,0.12) {$Q_{v1}$};
 \node[above] at (11.6,0.12) {$Q_v$};
 \fill (3.25,0) circle (2.1pt);
 \fill (8.75,0) circle (2.1pt);
 \node[below] at (3.25,-0.1) {$c_{v0}$};
 \node[below] at (8.75,-0.1) {$c_{v1}$};
 \foreach \x in {1.4,5.05,6.95,10.6}
   {\node[below] at (\x,-0.1) {$*$};}
 \draw[<->] (2.6,-1.05)--(3.9,-1.05);
 \node[below] at (3.25,-1.05) {$P_{v0}$};
 \draw[<->] (8.1,-1.05)--(9.4,-1.05);
 \node[below] at (8.75,-1.05) {$P_{v1}$};
\end{tikzpicture}
\caption{An active atom $c_v$ is replaced by a central band $J_v$ and two child atoms $c_{vi}$, inserted guards are $L_{vi},Q_{vi}$. Thickened segments are permanent bands. Further operations will be confined to corridors $P_vi$.}
\label{fig:split}
\end{figure}

We are now ready to prove the statement (ii) of Theorem \ref{thm:main}.

\begin{proposition}\label{prop:widom}
Let $\mathsf{E}$ be given by \eqref{eq:E-description}, then $\Omega=\C\setminus \mathsf{E}$ satisfies the Widom condition \eqref{eq:widom-definition}.
\end{proposition}

\begin{proof}
In order to establish the Widom condition, we need to bound the maximum in each finalized gap. For the operations applied to the root vertex, as shown in the schematic Figure 1, four gaps are finalized. At the end points of these gaps, the Green function vanishes. There is a maximum in each of the four gaps, denote the biggest maximum by $B.$ Assume that we are now at generation $n$. Then after this generation, each of the $2^n$ vertices generates four finalized gaps. With the chosen scales and by the established estimate \eqref{eq:green-invariant}, we have 
\[\sum_{I\in\mathcal{G}(E)}\max_{I}G_{\mathsf{E}}\leq 4B+4\sum_{n\geq 1}2^n\gamma_n.\]
By the choice of $\gamma_n=2^{-3n}$, the Widom condition holds.
\end{proof}

The construction remains within the class of products considered
in~\cite[Section 4]{NVY2007}. To be more specific, we have

\begin{proposition}\label{prop:infinite-product}
For $z\in\C\setminus[-1,1]$, the Cauchy transform of the limiting measure has the  product representation
\begin{equation}\label{eq:infinite-product}
 \mathcal{C}_\mu(z)=-\frac1{\sqrt{z^2-1}}
 \prod_{j\ge1}\frac{z-x_j}{\sqrt{(z-a_j)(z-b_j)}}.
\end{equation}
The product converges locally uniformly on that domain, with each
factor normalized to one at infinity. Its analytic continuation
across each bounded gap equals $\mathcal{C}_\mu$ there.
\end{proposition}

\begin{proof}
With branches fixed by normalization at infinity,
\begin{equation}\label{eq:log-factor}
 \log\frac{z-x_j}{\sqrt{(z-a_j)(z-b_j)}}
 =\frac12\int_{a_j}^{x_j}\frac{\dd t}{t-z}
  -\frac12\int_{x_j}^{b_j}\frac{\dd t}{t-z}.
\end{equation}
For any given compact set $F\subset \C\setminus [-1,1]$, there exists $d>0$ such that the distance between $F$ and $[-1,1]$ is at least $d.$ It follows that the absolute value of \eqref{eq:log-factor} is bounded by $\frac{b_j-a_j}{2d}$. Since the total length of all gaps is bounded, it follows that $\log \mathcal{C}_\mu$ (thus $\mathcal{C}_\mu$) is finite and well-defined on $\C\setminus [-1,1]$.

Notice that every operation in each generation is realized by Lemma~\ref{lem:splitting} and Lemma ~\ref{lem:guards}, i.e. multiplying the finite stage density $R_n(z)$ by some factors. Therefore, each finite stage density takes the form \eqref{eq:finite-product}. We need to show that $R_n(z)\to \mathcal{C}_\mu(z)$ locally uniformly.
At generation $n$, all gaps outside each $N_v,v\in\mathcal{T},|v|=n$ are finalized. The forth coming gaps generated by next-generation operations lie inside the pockets $N_v$. Let $P_v$ be the corridor associated to the vertex $v$ and denote $U_n=\cup_{|v|=n}P_v$. By the definition of $A_n$, we have $U_n\subset A_{n-1}$. Therefore $|U_n|\to 0$ as $n\to\infty.$ We compute that 
\[\sup_{z\in F}|\log \mathcal{C}_\mu(z)-\log R_n(z)|\leq \frac{|U_n|}{d}.\]
This proves the locally uniform convergence $R_n\to \mathcal{C}_\mu$.

Finally, $\mathcal{C}_\mu$ is already analytic on $\C\setminus \mathsf{E}$. The identity
just proved on the upper and lower half-planes identifies it as the
analytic continuation of the product across every finalized gap.
\end{proof}

\noindent\textbf{Acknowledgment}

\medskip

The author thanks Peter Yuditskii for suggesting the problem and encouraging him to complete their argument. The construction and the decisive finite-stage density formulas originate in the work of Nazarov-Volberg-Yuditskii \cite{NVY2007}. 
\medskip

\noindent\textbf{AI Usage Statement:}
No generative AI tool was used during the preparation of this manuscript.


\end{document}